\documentclass[smallcondensed,envcountsect,12pt]{svjour3} 
\smartqed  

\usepackage[utf8]{inputenc}
\usepackage{amsmath}
\usepackage{amsfonts}
\usepackage{amssymb}
\usepackage{enumerate}
\usepackage{hyperref}
\hypersetup{colorlinks=true,urlcolor=blue,citecolor=blue}
\usepackage{graphicx}
\usepackage{mathtools}
\usepackage{float}
\usepackage{xcolor}
\usepackage{framed}
\usepackage{fancyvrb}
\fvset{frame=single,framesep=1mm,fontfamily=courier,fontsize=\scriptsize,numbers=left,framerule=.3mm,numbersep=1mm}
\usepackage{caption}
\usepackage{subcaption}

\usepackage{graphicx} 
\usepackage{color} 
\usepackage{transparent} 
\usepackage{import}

\spnewtheorem{fact}{Fact}[section]{\bf}{\it}

\DeclareMathOperator{\prox}{prox}
\DeclareMathOperator{\vspan}{span}
\DeclareMathOperator*{\argmin}{arg\,min}

\DeclareMathOperator*{\dom}{dom}
\DeclareMathOperator*{\range}{range}

\DeclareMathOperator*{\cone}{cone}
\DeclareMathOperator*{\gph}{gph}

\DeclareMathOperator*{\cl}{cl}
\DeclareMathOperator*{\intr}{int}

\DeclareMathOperator{\prob}{prob}
\newcommand{\setto}{\rightrightarrows} 

\newcommand{\map}[3]{#1:\,#2\rightarrow #3\,}

\newcommand{\paren}[1]{\left(#1\right)}
\newcommand{\ecklam}[1]{\left[#1\right]}
\newcommand{\klam}[1]{\left\{#1\right\}}
\newcommand{\set}[2]{\left\{#1\,\left|\,#2\right.\right\}}

\colorlet{mapleCode}{red}
\colorlet{mapleMath}{blue}

\usepackage{marvosym} 
\newcommand{\envelope}{(\raisebox{-.5pt}{\scalebox{1.45}{\Letter}}\kern-1.7pt{\hspace{0.5ex}})}

\begin{document}

\author{Florian~Lauster \and D.~Russell~Luke \and Matthew~K.~Tam}
\authorrunning{F. Lauster \and D. R. Luke \and M. K. Tam}

\institute{F. Lauster \and D. R. Luke \and M. K. Tam \envelope \at
        Institut f\"ur Num. und Angew. Mathematik,
		Universit\"at G\"ottingen, 37083 G\"ottingen, Germany.\\
		\email{\href{mailto:m.tam@math.uni-goettingen.de}{m.tam@math.uni-goettingen.de}} \\[0.5em]
		F. Lauster\\
		\email{\href{mailto:florian.lauster@stud.uni-goettingen.de}{florian.lauster@stud.uni-goettingen.de}}	
		\\[0.5em]
		D.R. Luke\\
		\email{\href{mailto:r.luke@math.uni-goettingen.de}{r.luke@math.uni-goettingen.de}}
		}
		
\title{Symbolic Computation with Monotone Operators}

\date{\today\\[1.25em]
      {\normalsize\emph{Dedicated to the memory of Jonathan Michael Borwein}}}

\def\subclassname{{\bfseries Mathematics Subject Classification (2010)}\enspace}
\def\makeheadbox{\relax}

\maketitle

\begin{abstract}
	We consider a class of monotone operators which are appropriate for symbolic representation 
and manipulation within a computer algebra system. Various structural properties of the class 
({\em e.g.,} closure under taking inverses, resolvents) are investigated as well as the role 
played by maximal monotonicity within the class. In particular, we show that there is a 
natural correspondence between our class of monotone operators and the subdifferentials 
of convex functions belonging to a class of convex functions deemed suitable for symbolic 
computation of Fenchel conjugates which were previously studied by Bauschke \& von 
Mohrenschildt and by Borwein \& Hamilton. A number of illustrative examples utilizing 
the introduced class of operators are provided including computation of proximity operators,  
recovery of a convex penalty function associated with the hard thresholding operator, 
and computation of superexpectations, superdistributions and superquantiles with specialization to risk measures. 
\end{abstract}



\section{Introduction}
 The \emph{Fenchel conjugate} and the \emph{subdifferential} of a function are two objects of 
fundamental importance in convex analysis. For this reason, software libraries or packages which 
have the ability to compute and manipulate such objects easily are a valuable edition to the convex 
analyst's toolkit.  In the spirit of \emph{experimental mathematics} \cite{BaileyBorweinGirgensohn03,BaileyBorwein03,BaileyBorwein05,BBCGLM}, such software also 
enables researchers to use the machinery of convex analysis to test ideas and look for 
patterns. There are also potential pedagogical uses if one believes, as we do, that nonsmooth analysis could and 
should be a part of the traditional ``calculus'' cannon taught to high school and beginning Bachelor's level students.

 There are at least two possible paradigms for computation which can be followed for the development of such a 
software library, namely, computations can be done \emph{numerically} or \emph{symbolically}. Roughly 
speaking, the former involves numerical evaluation of the object under consideration on a grid of points in the 
ambient space whilst the latter involves the  manipulation objects through \emph{symbolic expressions} 
with a \emph{Computer Algebra System (CAS)}. We focus on the second approach, the ``symbolic paradigm". 
For further details regarding numerical convex analysis,  we refer the reader to 
\cite{lucet1997faster,lucet2010shape,GarLuc14}.

It is not too difficult to imagine that there are convex functions which, if not impossible, are too complex to represent and manipulate symbolically.  Nevertheless, by restricting oneself to a suitable class of convex functions, a great deal can still be accomplished. Such a framework for symbolic convex analysis was proposed by Bauschke \& von Mohrenschildt \cite{bauschke1997fenchel,bauschke2006symbolic} for functions on the real line and an extension which could handle a many-dimensional setting was later proposed by Borwein \& Hamilton \cite{hamilton2005symbolic,borwein2009symbolic}. By ``suitable class" we mean a class of functions which are representable 
by an appropriate data-structure, are closed under operations such as Fenchel conjugation and are sufficiently generic so as to capture many important examples.

 As the subdifferential of a convex function is a monotone operator, it is natural to ask what class of monotone 
operators is suitable for symbolic computation as well as their relationship to the ``suitable class" of convex functions 
studied in \cite{bauschke1997fenchel,bauschke2006symbolic,hamilton2005symbolic,borwein2009symbolic}. 
To the best of our knowledge, little has been done in this direction with the aforementioned works not focusing on the structure of the underlying monotone operators directly.
In this work, we propose and study such a 
class of monotone operators which are suitable for implementation within a CAS. Among our main results, we 
prove that there is a natural correspondence between our class of monotone operators and the subdifferentials 
of the functions studied by Bauschke \& von Mohrenschildt (Theorem~\ref{th:F and T}).  
We show that a consequence of this is that the class is closed under addition, scalar multiplication, taking 
inverses and taking resolvents (Proposition~\ref{prop:T operator properties}).  We demonstrate the application of 
this class of operators on several illustrative examples in Section~\ref{s:examples}.

\section{Preliminaries}
Our notation and terminology are standard and can be found, for instance, in \cite{borweinvanderwerff2010convex} 
and \cite{VA}.  Since this work concerns computer implementations, we restrict our attention to 
$\mathbb{R}^n$ equipped with standard dot-product, denoted $\langle\cdot,\cdot\rangle$. 
For quick reference, we list some well-known facts from convex analysis that will be important later. 
 
The  \emph{(effective) domain} of a function ${f\colon\mathbb{R}^n\to[-\infty,+\infty]}$  is the set 
  $\dom f:=\{x\in\mathbb{R}^n: |f(x)|<+\infty\}.$  We will be interested in {\em proper (not everywhere infinite and 
nowhere equal to $-\infty$), lower semi-continuous (lsc), convex functions}. The \emph{sub-differential} of a convex function is the set-valued mapping 
$\partial f:\mathbb{R}^n\setto\mathbb{R}^n$ given by
   $$\partial f(\overline{x}) := \begin{cases}
                         \{\phi\in\mathbb{R}^n:\langle\phi,x-\overline{x}\rangle \leq f(x)-f(\overline{x})\} &\overline{x}\in\dom f,\\
                         \emptyset & \overline{x}\not\in\dom f. \\
                         \end{cases}$$
  The \emph{Fenchel conjugate} of $f$ is the function $f^*:\mathbb{R}^n\to[-\infty,+\infty]$ defined by
   $$ f^*(y) := \sup_{x\in\mathbb{R}^n}\{\langle y,x\rangle-f(x)\}. $$
 The subdifferentials of a function and its Fenchel conjugate are inversely related. 
 \begin{fact}[{\cite[Prop.~4.4.5]{borweinvanderwerff2010convex}}]\label{fact1}
 	Let $f:\mathbb{R}^n\to(-\infty,+\infty]$ be a function with $\overline{x}\in\dom f$. 
 	If $\overline{v}\in\partial f(\overline{x})$ then $\overline{x}\in\partial f^*(\overline{v})$. 
 	Conversely, if $f$ is a convex function which is lsc at $\bar{x}$ and
$\overline{v}\in\partial f^*(\overline{x})$, then $\overline{v}\in\partial f(\overline{x})$.
 \end{fact}  

 Let $T:\mathbb{R}^n\setto\mathbb{R}^n$ be a set-valued map. The \emph{domain} of $T$ 
is the set $\dom T:=\{x\in\mathbb{R}^n:Tx\neq\emptyset\}$, and 
the \emph{graph} of $T$ is the set $\gph T:=\{(x,y)\in\mathbb{R}^n\times\mathbb{R}^n:y\in Tx\}$. 
Recall that $T$ is \emph{monotone} if 
   \begin{equation}\label{eq:monotone defn}
	   \begin{rcases*}
	     (x,x^+)\in\gph T \\
	     (y,y^+)\in\gph T \\
	   \end{rcases*} \implies
	   \langle x-y,x^+-y^+\rangle\geq 0.
   \end{equation}
 If the inequality in \eqref{eq:monotone defn} is strict whenever $x\neq y$, then $T$ is said to be \emph{strictly monotone}. If $T$ is monotone and there exists no monotone operator whose graph properly contains the graph of $T$, then $T$ is said to be \emph{maximal monotone}.
 
 Next we recall a notion stronger than monotonicity. An operator $T:\mathbb{R}^n\setto\mathbb{R}^n$ 
is said to be \emph{cyclically monotone} if, for every $n\geq 2$, it holds that
   \begin{equation*}
   \begin{rcases*}
   (x_1,x^+_1)\in\gph T \\
   \hphantom{(x_1,x^+_1)} \vdots \\
   (x_n,x_n^+)\in\gph T \\
   \hphantom{(x_n} x_{n+1} = x_1 \\
   \end{rcases*} \implies
   \sum_{i=1}^n\langle x_i-x_{i+1},x^+_i\rangle\geq 0.
   \end{equation*}
 Analogously, if $T$ is cyclically  monotone and there exists no cyclically monotone operator whose graph 
properly contains the graph of $T$, then $T$ is said to be \emph{maximal cyclically monotone}.
 
 The family of maximal cyclically monotone operators can be characterized as  
 the subdifferentials of proper, lsc, convex functions.
 \begin{fact}[Rockafellar \cite{Rockafellar70}]\label{fact:cyclically monotonicity}
 	Let $T:\mathbb{R}^n\setto\mathbb{R}^n$. Then $T$ is maximal cyclically monotone if and only 
if there exists a proper, lsc, convex function $f:\mathbb{R}^n\to(-\infty,+\infty]$ such that $T=\partial f$. 
 \end{fact}
 Regarding Fact~\ref{fact:cyclically monotonicity}, we remark that cyclic monotonicity is stronger than mere monotonicity. However, on the real-line ({\em i.e.,} $\dom T=\mathbb{R}$), the two notions coincide \cite[Th.~22.18]{bauschke2011convex}.
 
 In general, the sum of two maximal monotone operators need not be maximal unless an appropriate 
constraint qualification is satisfied. In the following Fact, we give one such example.
 \begin{fact}[Maximal monotonicity of sums {\cite[Th.~24.3]{bauschke2011convex}}]\label{fact:maximal monotone sum}
 	Let $T_1,T_2:\mathbb{R}^n\setto\mathbb{R}^n$ be maximal monotone such that 
 	 \begin{equation}\label{eq:maximal monotone sum cq}
 	  \cone(\dom T_1-\dom T_2)=\cl\vspan(\dom T_1-\dom T_2).
 	 \end{equation}
 	 Then $T_1+T_2$ is maximal monotone. In particular, \eqref{eq:maximal monotone sum cq} 
	  holds whenever $0\in\intr(\dom T_1-\dom T_2)$.
 \end{fact}

Let $S$ be a subset of $\mathbb{R}^n$. Recall that $S$ is said to be an \emph{$m$-dimensional simplex} 
if there exist a set of $m+1$ affinely independent points whose convex hull equals $S$. We say that $S$ 
is \emph{locally simplicial}, in the sense of 
\cite{rockafellar1970convex}, if for each $x\in S$ there exists a finite collection of simplicies 
$S_1,\dots,S_m$ contained in $S$ such that, for some neighborhood $U$ of $x$, 
  $$U\cap \left(S_1\cup\dots\cup S_m\right) = U\cap S.$$
Examples of locally simplicial sets are line segments, polyhedral convex sets, and relatively open convex sets.
\begin{fact}[Continuity on the effective domain 
{\cite[Th.~10.2]{rockafellar1970convex}}]\label{fact:cont local simplicial}
	Let $f$ be a convex function on $\mathbb{R}^n$, and let $S$ be any locally simplicial subset of $\dom f$. 
Then $f$ is upper semi-continuous relative to $S$. In particular, if $f$ is lsc, then $f$ is continuous relative to $S$.
\end{fact}

As a consequence of Fact~\ref{fact:cont local simplicial} and the fact that every convex subset of $\mathbb{R}$ is an interval (and hence simplicial), it follows that every convex lsc function on $\mathbb{R}$ is necessarily continuous on its domain.

\section{Operators in One Dimension}\label{sec:operators in 1D}
 
 Before turning our attention to monotone operators, we recall first the class of convex functions 
originally studied in \cite{bauschke1997fenchel,bauschke2006symbolic,hamilton2005symbolic,borwein2009symbolic}.  
The \emph{Maple} package \emph{Symbolic Convex Analysis Toolbox (SCAT)} 
\cite{borwein2009symbolic} implements precisely this class of functions. Since the 
definition of the class is recursive in the dimension of the underlying domain, it is essential to 
understand the one-dimensional case first. 

Recall that a function $f:\mathbb{R}^n\to[-\infty,+\infty]$ is \emph{strictly convex} if
  \begin{equation*}
  	f(\lambda x+(1-\lambda)y) <\lambda f(x)+(1-\lambda)f(y),
  \end{equation*}
for all $x,y\in\dom f$ with $x\neq y$ and for all $\lambda\in(0,1)$.
 
\begin{definition}[$\mathcal{F}$-functions {\cite{bauschke1997fenchel,bauschke2006symbolic}}]\label{def:F functions}
	For a set of finitely many points $A=\{a_i\}_{i=1}^m$ satisfying 
	\begin{equation}\label{eq:set A}
		a_0=-\infty<a_1<\dots<a_{m-1}<a_m=+\infty,
	\end{equation} 
	we say a function $f:\mathbb{R}\to(-\infty,+\infty]$ belongs to $\mathcal{F}(A)$ if:
	\begin{enumerate}[(a)]
		\itemsep0em 
		\item\label{def:F functions a} $f$ is closed and convex;
		\item\label{def:F functions b} $f$ is continuous on its effective domain; and
		\item\label{def:F functions c} the restriction of $f$ to the interval $(a_i,a_{i+1})$ is either
              (i)~affine,
              (ii)~strictly convex and differentiable, or
              (iii)~identically equal to $+\infty$.
	\end{enumerate}
	The class of functions $\mathcal{F}$ is the union of $\mathcal{F}(A)$ over all finite sets of 
	points $A$ satisfying \eqref{eq:set A}    
\end{definition}

 Recalling that a function is closed ({\em i.e.,} its epigraph is a closed set) if and only if it is lsc, 
we observe that Condition~\eqref{def:F functions a} is 
equivalent to requiring that functions in $\mathcal{F}$ be either proper, lsc and convex, or 
identically equal to $+\infty$. 
Moreover, as a convex function is continuous on the relative interior of its domain 
\cite[Th.~10.1]{rockafellar1970convex}, 
the only place where Condition~\eqref{def:F functions b} can play a role is at boundary points 
of the effective domain.
 
 \begin{remark}
 	The presentation of Definition~\ref{def:F functions} given here differs slightly from the version given 
in \cite{bauschke1997fenchel,bauschke2006symbolic} in that we introduce the set $\mathcal{F}$ through the union 
of the sets $\mathcal{F}(A)$ rather than directly.
 \end{remark}

 One of the most important properties of $\mathcal{F}$-functions is that the class is closed under 
 Fenchel conjugation. This is ensures that a data-structure designed to represent functions belonging to $\mathcal{F}$ is also able to represent their conjugates. This closure property was noted in \cite{bauschke2006symbolic} without proof. We shall return to this topic later where our soon to be introduced class of monotone operator to furnish a convenient proof. Another important property of the subdifferentials of $\mathcal{F}$-functions is that they may be expressed explicitly  in terms of their gradient, when this exists. 
 \begin{proposition}[Computing $\mathcal{F}$-subdifferentials {\cite[\S2.1.2]{borwein2009symbolic}}]
\label{prop:computing subdiff F}
  Suppose $f\in\mathcal{F}(A)$ for $A=\{a_i\}_{i=1}^m$, and let $f|_i$ (resp. $f'|_i$) denote the 
restriction of $f$ (resp. $f'$) to the interval $(a_i,a_{i+1})$. Then $\partial f$ can be piecewise defined 
according to the following three cases.
  \begin{enumerate}[(a)]
    \item\label{prop:computing subdiff F a} If $x\not\in\dom f$ then $\partial f(x_0)=\emptyset$;
    \item\label{prop:computing subdiff F b} If $x\in\intr(\dom f)$ then
       $$ \partial f(x) = \left[\lim_{y\uparrow x}f'(y),\,\lim_{y\downarrow x}f'(y) \right]; $$
    \item\label{prop:computing subdiff F c} If $x\in\dom f\setminus\intr(\dom f)$ then $x=a_i$ 
for some $i\in\{1,2,\dots,m-1\}$. In this case
       $$\partial f(a_i)  = \begin{cases}
                              (-\infty,\,+\infty) & \text{if }f|_{i-1}=\infty=f|_i, \\
                              (-\infty,\,\lim_{y\downarrow a_i}f|_i'(y)] & \text{if }f|_{i-1}=\infty\neq f|_i, \\
                              [\lim_{y\uparrow a_i}f|_{i-1}'(y),\,+\infty) & \text{if }f|_{i-1}\neq\infty=f|_i. \\          
                            \end{cases}$$
  \end{enumerate}
 \end{proposition}
 
  As we have already seen, the subdifferential of proper, lsc, convex function is always a 
maximal (cyclically) monotone operator (Fact~\ref{fact:cyclically monotonicity}). Thus, in light of the 
above proposition, we collect some of the finer monotonicity properties of the subdifferentials of 
$\mathcal{F}$-functions. The following lemma will simplify the proof of Proposition~\ref{prop:subdiff F}.

  \begin{lemma}\label{lem:subdiff f}
  	Let $f\in\mathcal{F}(A)$ be a proper function where $A=\{a_i\}_{i=0}^m$. Then the 
  	restriction of $\partial f$ to the interval $(a_i,a_{i+1})$ is either
  	 (i)~single-valued and constant,
  	 (ii)~single-valued, continuous and strictly monotone, or
  	 (iii)~identically equal to the empty-set.
  \end{lemma}
  \begin{proof}
  	Consider the restriction of $f$ to the open interval $(a_i,a_{i+1})$. We distinguish three cases based on Definition~\ref{def:F functions}\eqref{def:F functions c}. 
  	 (i)~If $f$ is affine on $(a_i,a_{i+1})$ then $f'$ is single-valued and constant.
  	 (ii)~If $f$ is	differentiable on $(a_i,a_{i+1})$ then $f'$ is continuous on 
  	$(a_i,a_{i+1})$ \cite[Thm.~25.5.1]{rockafellar1970convex}, and if $f$ is strictly convex on $(a_i,a_{i+1})$ then $f'$ is strictly monotone on $(a_i,a_{i+1})$, by \cite[Exer.~2.1.14]{borweinvanderwerff2010convex} and \cite[Thm. 12.17]{VA}.
  	(iii)~Otherwise, by virtue belonging to 
  	$\mathcal{F}$, $f$ must be  identically equal to 
  	$+\infty$ on $(a_i,a_{i+1})$ and, by definition, its subdifferential is identically 
  	equal to the empty-set on $(a_i,a_{i+1})$.\qed
  \end{proof}

 \begin{proposition}[Structure of $\mathcal{F}$-subdifferentials]\label{prop:subdiff F}
 	Let $f\in\mathcal{F}(A)$ be a proper function where
$A=\{a_i\}_{i=0}^m$. The following assertions hold.
 	\begin{enumerate}[(a)]
        \itemsep0em 
 		\item\label{prop:subdiff F b} The restriction of $\partial f$ to the interval $(a_i,a_{i+1})$ is 
		  either (i)~single-valued and constant, (ii)~single-valued, continuous and strictly monotone, or 
		  (iii)~identically equal to the empty-set.
 		\item\label{prop:subdiff F c} For any $x\in\dom f\setminus\intr(\dom f)$, $x=a_i$ for some $i\in\{1,m-1\}$ and
 		  $$\begin{cases}
 		   \,\hphantom{\max}\partial f(a_i) = (-\infty,\,+\infty) & \text{if }f|_{i-1}=\infty=f|_i, \\
 		   \max\partial f(a_i) = \lim_{x\downarrow a_i}f|_i'(x) & \text{if }f|_{i-1}=\infty\neq f|_i, \\
 		   \,\min\partial f(a_i) = \lim_{x\uparrow a_i}f|_{i-1}'(x) & \text{if }f|_{i-1}\neq\infty=f|_i; \\          
 		  \end{cases}$$
 		where, by convention, $\min\emptyset=+\infty$ and $\max\emptyset=-\infty$.
 	\end{enumerate}
 \end{proposition}
\begin{proof}
 \eqref{prop:subdiff F b}:~Follows by applying Lemma~\ref{lem:subdiff f} to each interval $(a_i,a_{i+1})$. 
 \eqref{prop:subdiff F c}:~This is a direct consequence of Proposition~\ref{prop:computing subdiff F}\eqref{prop:computing subdiff F c}.\qed
\end{proof}
 
 Motivated by Propositions~\ref{prop:computing subdiff F} and \ref{prop:subdiff F}, we define the following class of monotone 
operators. We shall show that, in particular, it contains all subdifferentials of $\mathcal{F}$-functions.
 
 \begin{definition}[$\mathcal{T}$-operators]\label{def:T operators}
 	For a set of finitely many points $B=\{b_i\}_{i=0}^l$ satisfying
 	\begin{equation}\label{eq:set B}
 	b_0=-\infty<b_1<\dots<b_{l-1}<b_l=+\infty,
 	\end{equation}
 	we say a set-valued operator $T:\mathbb{R}\setto\mathbb{R}$ belongs to $\mathcal{T}(B)$ if there 
	exists a maximal monotone extension $\widetilde{T}$ of $T$ such that the restriction 
	 $\widetilde{T}$ to each interval $(b_i,b_{i+1})$ is either
 	\begin{enumerate}[(i)]
 	\item single-valued and constant;
 	\item single-valued, continuous and strictly monotone; or 
 	\item identically equal to the empty-set.
 	\end{enumerate}
	The class of operators $\mathcal{T}$ is the union of $\mathcal{T}(B)$ over all finite sets of points 
      $B$ satisfying \eqref{eq:set B}.
 \end{definition} 

  The proposition which soon follows establishes that the class of $\mathcal{T}$-operators is well-suited for symbolic manipulation. Moreover, as a monotone operator can have only countably many discontinuities in its domain, the restriction to monotone operators possessing at most finitely many discontinuities is still quite general.
 
\begin{remark}[$\mathcal{T}$-operators at points of discontinuity]
	Let $T\in\mathcal{T}$ with maximal monotone extension $\widetilde{T}\in\mathcal{T}(B)$  
where $B=\{b_i\}_{i=1}^l$. From the definition of $\mathcal{T}$-operators, the only possible 
points of discontinuity of $T$ are the points in $B$. At a point $b_i\in\intr(\dom\widetilde{T})$ for 
some $i\in\{1,\dots,l-1\}$, the restriction of  $\widetilde{T}$ to either of the open intervals 
$(b_{i-1},b_i)$ and $(b_i,b_{i+1})$ is continuous and hence, by monotonicity, both of the 
limits $\lim_{x\uparrow b_i}\widetilde{T}(x)$ and $\lim_{x\downarrow b_i}\widetilde{T}(x)$ 
are finite. We therefore have that
	  $$ T(b_i)\subseteq \widetilde{T}(b_i)=\left[\lim_{x\uparrow b_i}T(x),\,\lim_{x\downarrow b_i}T(x)\right], $$
	where the equality holds due to outer semi-continuity of $\widetilde{T}$ 
	\cite[\S4.2]{burachik2007set}.
\end{remark}

 The following theorem shows that all of the most important closure properties hold for the 
class of $\mathcal{T}$-operators.
\begin{proposition}[Properties of $\mathcal{T}$-operators]\label{prop:T operator properties}
	The following assertions hold.
	\begin{enumerate}[(a)]
		\item\label{prop:T operator properties a} If $T\in\mathcal{T}$ and $\lambda\geq 0$ 
then $\lambda T\in\mathcal{T}$.	    
		\item\label{prop:T operator properties b} If $T_1,T_2\in\mathcal{T}$ then 
$T_1+T_2\in\mathcal{T}$. 
	    \item\label{prop:T operator properties c} If $T\in\mathcal{T}$ then 
$T^{-1}\in\mathcal{T}$.
	    \item\label{prop:T operator properties d} If $T\in\mathcal{T}$ and $\lambda>0$ 
then $(I+\lambda T)^{-1}\in\mathcal{T}$.
        \item\label{prop:T operator properties e} If $T_1,T_2\in\mathcal{T}$ then 
        $(T_1^{-1}+T_2^{-1})^{-1}\in\mathcal{T}$. 
	\end{enumerate}
\end{proposition}	
\begin{proof}
	\eqref{prop:T operator properties a}:~Let $\widetilde{T}\in\mathcal{T}$ be a maximal 
monotone extension of $T$ and $\lambda\geq 0$. Then $\lambda\widetilde{T}\in\mathcal{T}$ 
and, moreover, $\lambda\widetilde{T}$ is a maximal monotone extension of $\lambda T$.
	
	\eqref{prop:T operator properties b}:~Define $T:=T_1+T_2$ and let $\widetilde{T}_1$ 
and $\widetilde{T}_2$ denote maximal monotone extensions respectively of $T_1$ and $T_2$ 
contained in  $\mathcal{T}$. Setting $\widetilde{T}:=\widetilde{T}_1+\widetilde{T}_2$, 
we therefore have that
	\begin{equation}\label{eq:dom T}
		\dom\widetilde{T}=\dom\widetilde{T}_1\cap\dom\widetilde{T}_2 
\supseteq \dom T_1\cap\dom T_2=\dom T.
	\end{equation}
	We now distinguish three cases based on $\dom\widetilde{T}$.
	\begin{enumerate}[(i)]
	 \item Suppose $\dom\widetilde{T}=\emptyset$. Then, using \eqref{eq:dom T}, 
it follows that $\dom T=\emptyset$. As the empty relation is trivially contained in $\mathcal{T}$, 
we have that $T\in\mathcal{T}$.
	
	 \item Suppose $\dom\widetilde{T}\neq\emptyset$ but $\intr(\dom\widetilde{T})=\emptyset$. 
Since $\dom\widetilde{T}$ is the intersection of two convex sets, $\dom\widetilde{T}_1$ and 
$\dom\widetilde{T}_2$, it follows that $\dom\widetilde{T}$ is a singleton, say, 
$\dom\widetilde{T}=\{x_0\}$. In this case, the operator
	$$ x\mapsto \begin{cases}
	(-\infty,+\infty) & \text{if }x=x_0, \\
	\emptyset         & \text{otherwise} \\
	\end{cases}$$
	defines a maximal monotone extension of $\widetilde{T}$, and hence also defines a 
maximal monotone extension of $T$, which is contained in $\mathcal{T}$.
	
	 \item Suppose $\intr(\dom\widetilde{T})\neq\emptyset$. Then 
$0\in\intr(\dom\widetilde{T}_1-\dom\widetilde{T}_2)$ and hence, by   
Fact~\ref{fact:maximal monotone sum},  the extension $\widetilde{T}$ is 
maximal monotone. 
Let $\{b_i\}_{i=1}^l$ denote the union of the sets of breakpoints for $\widetilde{T}_1$ 
and $\widetilde{T}_2$, provided by Definition~\ref{def:T operators}. To see that 
$\widetilde{T}\in\mathcal{T}$, observe that the restriction of $\widetilde{T}$ to each 
open interval $(b_i,b_{i+1})$ is either single-valued and continuous, single-valued and 
strictly monotone, or identically equal to the empty-set.
	\end{enumerate}
	
	\eqref{prop:T operator properties c}:~Let $\widetilde{T}\in\mathcal{T}(B)$ be a maximal monotone extension of $T$. Since $\widetilde{T}^{-1}$ is a maximal monotone extension of $T^{-1}$, it suffices to show that $\widetilde{T}^{-1}\in\mathcal{T}$. To this end, observe that
	\begin{equation}\label{eq:dom Tinv}
	\dom\widetilde{T}^{-1}=\range\widetilde{T} = \left(\bigcup_{i=0}^l \widetilde{T}((b_i,b_{i+1}))\right)\cup\left(\bigcup_{i=1}^{l-1}\widetilde{T}(b_i)\right),
	\end{equation}
	where we denote $\widetilde{T}((b_i,b_{i+1})):=\{y\in\widetilde{T}(x):x\in(b_i,b_{i+1})\}$. 
	Both $\widetilde{T}$ and $\widetilde{T}^{-1}$ are maximal monotone and closed convex-valued 
	\cite[Exerc.12.8]{VA}.  To show that $\widetilde{T}^{-1}\in\mathcal{T}$, it suffices to show that on 
	each piece of its domain specified by  \eqref{eq:dom Tinv} which is not a singleton, that 
	$\widetilde{T}^{-1}$ is single-valued and either constant, or strictly monotone and hence continuous 
	by maximal monotonicity. To see this, we distinguish the following cases, using the fact that $\widetilde{T}\in\mathcal{T}$.
	
	\begin{enumerate}[(i)]    	
		\item\label{c:case 1} Consider a piece in \eqref{eq:dom Tinv} of the form $\widetilde{T}((b_i,b_{i+1}))$. 
		  There are thus  two possibilities:
		\begin{enumerate}[(I)]
			\item \emph{$\widetilde{T}$ is single-valued and constant with value $c$ on $(b_{i},b_{i+1})$}: 
			In this case $\widetilde{T}^{-1}(c)$ is a closed interval containing $(b_{i},b_{i+1})$. 
			
			\item \emph{$\widetilde{T}$ is single-valued, continuous and strictly monotone 
			on $(b_{i},b_{i+1})$}: In this case, $\widetilde{T}^{-1}$ is single-valued, continuous and strictly monotone on $U=\widetilde{T}( (b_i,b_{i+1}) )$ and, moreover, $U$ is an interval \cite[Th.~5.11.14]{vakil2011real}.
		\end{enumerate}
		
		\item\label{c:case 2} Next consider a piece in \eqref{eq:dom Tinv} of the form $\widetilde{T}(b_i)$ where $b_i\in\dom\widetilde{T}$. Then $b_i\in \widetilde{T}^{-1}(y)$ for all $y\in\widetilde{T}(b_i)$ where $\widetilde{T}(b_i)$ is a closed convex set due to the maximal monotonicity of $\widetilde{T}$. If $\intr\widetilde{T}(b_i)=\emptyset$ then $\widetilde{T}(b_i)$ is a singleton and there is 
		nothing further to prove.  Suppose, then, that the open interval $U:=\intr\widetilde{T}(b_i)$ is non-empty. 
		There are two possibilities.
		\begin{enumerate}[(I)]
			\item \emph{$\widetilde{T}^{-1}$ is single-valued on $U$}:   Then $\widetilde{T}^{-1}(y)=\{b_i\}$ for all $y\in U$. 
			
			\item \emph{$\widetilde{T}^{-1}$ is multi-valued on $U$}:m   Then there exist points $y_0\in U$ and 
			$x_0\neq b_i$ such that $x_0\in\widetilde{T}^{-1}(y_0)$ and 
			$y_0> \underline{y}\geq \inf\widetilde{T}(b_i)$. 
			For convenience, we assume that 
			$x_0<b_i$; an analogous argument applies when $x_0>b_i$. Since $\widetilde{T}^{-1}$ 
			is closed and convex-valued,  $[x_0,b_i]\subseteq\widetilde{T}^{-1}(y_0)$ and hence 
			$y_0\in\widetilde{T}(x)$ for all $x\in[x_0,b_i]$. Since $\widetilde{T}\in\mathcal{T}$, it must be 
			single-valued on $(b_{i-1},b_i)$,  hence $\widetilde{T}x=\{y_0\}$ for all $x\in (b_{i-1},b_i)$. 
			Since $\widetilde{T}$ is maximal monotone, it must hold that $y_0=\inf T(b_i)\leq \underline{y}$
			which is a contradiction.
		\end{enumerate}
		We conclude, therefore, that $\widetilde{T}^{-1}(y)=\{b_i\}$ for all $y\in \widetilde{T}(b_i)$.  
	\end{enumerate}
	Cases \eqref{c:case 1} and \eqref{c:case 2} together imply that $\widetilde{T}^{-1}\in \mathcal {T}$ which completes the proof of  \eqref{prop:T operator properties c}.

 \eqref{prop:T operator properties d}:~By \eqref{prop:T operator properties a} it 
		follows that $\lambda T\in\mathcal{T}$. Noting that the identity operator $I$ is a 
		maximal monotone operator contained in $\mathcal{T}$ with $\dom I=\mathbb{R}$, 
		\eqref{prop:T operator properties b} implies that $I+\lambda T\in\mathcal{T}$. The 
		result now follows from \eqref{prop:T operator properties c}.
		
 \eqref{prop:T operator properties e}:~This follows immediately from parts~\eqref{prop:T operator properties b} and \eqref{prop:T operator properties c}.\qed
\end{proof}

\begin{example}[Examples of $\mathcal{T}$-operators]
	Proposition~\ref{prop:T operator properties}\eqref{prop:T operator properties d} ensures that
 the \emph{resolvent} of any $\mathcal{T}$-operator belongs to $\mathcal{T}$. 
In particular, $\mathcal{T}$ contains all \emph{proximity mappings} of $\mathcal{F}$-functions. 
In other words, if $f\in\mathcal{F}$ and $\lambda>0$ then
	  $$\prox_{f}^\lambda:=
\argmin_{y\in\mathbb{R}}\left\{f(y)+\frac{1}{2\lambda}\|\cdot-y\|^2\right\}
=(I+\lambda \partial f)^{-1}\in\mathcal{T}.$$
    In particular, by considering the indicator function contained in $\mathcal{F}$, 
we see that $\mathcal{T}$ contains all \emph{projection operators} onto closed, 
convex subsets of $\mathbb{R}$.
\end{example}

\begin{remark}[Maximal monotone extensions of $T\in\mathcal{T}$]
  When defining the class of $\mathcal{T}$-operators in 
Definition~\ref{def:T operators}, one might have instead required an 
operator $T\in\mathcal{T}$  to be maximal monotone itself rather 
than its extension in $\mathcal{T}$. This approach has a significant 
shortcoming in that the empty relation is no longer in $\mathcal{T}$. 
Consequently, Proposition~\ref{prop:T operator properties}\eqref{prop:T operator 
properties b} no longer holds as can be seen by considering the sum 
two maximal monotone operators whose domains do not intersect.
\end{remark}

 The following theorem summarizes the connection between $\mathcal{F}$-functions and 
 $\mathcal{T}$-operators.
 \begin{theorem}\label{th:F and T}
 	If $f\in\mathcal{F}$ is proper then $\partial f$ is maximal monotone and belongs to 
 	$\mathcal{T}$. Conversely, if $T\in\mathcal{T}$ is maximal monotone then there exists a 
 	proper, lsc, convex function $f$ such that $T=\partial f$ and, moreover, any such function 
 	belongs to $\mathcal{F}$.
 \end{theorem}   
 \begin{proof}
 	If $f\in\mathcal{F}$ is a proper function then the fact that $\partial f$ is maximal 
 	monotone and belongs to $\mathcal{T}$ was already proven in Proposition~\ref{prop:subdiff F}.
 	
 	Conversely, let $T\in\mathcal{T}$ be a maximal monotone operator. 
 	By Fact~\ref{fact:cyclically monotonicity}, there exists at least one proper, lsc, convex 
 	function with subdifferential equal to $T$. Let $f$ denote any such function (which 
 	already satisfies Definition~\ref{def:F functions}\eqref{def:F functions a}). By Fact~\ref{fact:cont local simplicial}, $f$ is continuous on $\dom f$, that is, Definition~\ref{def:F functions}\eqref{def:F functions b} is satisfied. Finally, to show that $f$ satisfies Definition~\ref{def:F functions}\eqref{def:F functions c}, first recall that a convex function is differentiable at point in its domain if and only if its subdifferential 
 	is a singleton at the same point \cite[Th.~2.2.1]{borweinvanderwerff2010convex}. It follows 
 	that $f$ can be non-differentiable only if $T$ is multi-valued which happens at most at finitely many points. 
 	Consider the restriction of the function $f$ to an open interval on which it is differentiable. Then, 
 	as $T\in\mathcal{T}$, we have that $f'$ is either constant or strictly monotone on this interval. 
 	If $f'$ constant then $f$ is affine. Otherwise  $f'$ is strictly monotone and hence $f$ is strictly convex 
 	\cite[Th.~2.13]{VA}.\qed
 \end{proof}	  
  Note that Theorem~\ref{th:F and T} provides a pathway to symbolically computing a maximal monotone extension of a monotone operators $T\in\mathcal{T}$. First find function $f\in\mathcal{F}$ such that $\partial f=T$. A maximal extension of $T$ is then given by $\partial f$.

  We now return to the question closure of $\mathcal{F}$-function under the operation of Fenchel conjugation. We offer the following proof, which utilizes our class of monotone operators.
  \begin{proposition}[$\mathcal{F}$ is closed under Fenchel conjugation]\label{prop:conj F is F}
	If $f\in\mathcal{F}$ is proper then $f^*\in\mathcal{F}$.
  \end{proposition}	
  \begin{proof}
  	By Theorem~\ref{th:F and T}, $\partial f\in\mathcal{T}$. Combining Fact~\ref{fact1} with 
Proposition~\ref{prop:T operator properties}\eqref{prop:T operator properties c} shows that 
$\partial f^*=(\partial f)^{-1}\in\mathcal{T}$. Using Theorem~\ref{th:F and T} a second time yields $f^*\in\mathcal{F}$.\qed
  \end{proof}

Assumption~\eqref{def:F functions c} of Definition~\ref{def:F functions} is crucial for obtaining closedness of the family $\mathcal{F}$ under Fenchel conjugation. 
Specifically, the strict convexity assumption on non-constant pieces of the domain cannot be removed and replaced with mere differentiability. 
In fact, the following counter-example shows that this is the case even for infinitely differentiable convex functions.
\begin{example}[Necessity of finitely affine pieces]
Consider the convex function (see Figure~\ref{fig:example})
constructed from (unnormalized)
$C^\infty$ mollifying functions as defined follows:
\begin{equation}\label{eq:example}
f(x):=\int_0^x \int_0^y h(z)\,dz\, dy
\end{equation}
with
\begin{align*}
h(x)&:=\begin{cases}
             \psi\left( 2^{2n+1}x-1 \right)&\mbox{for }x\in[2^{-(2n+1)}, 2^{-2n}]\quad (n\in\mathbb{N}_0)\\
             0 &\mbox{for }x\in(2^{-(2n+2)}, 2^{-(2n+1)}) \quad(n\in\mathbb{N}_0)\\
             0 &\mbox{for }x\in(-\infty, 0]\cup(1,\infty),
            \end{cases}\\          
\intertext{where $\psi$ denotes the mollifying function given by}            
	\psi(x)&:=\begin{cases} 
           \exp\left(-\frac{1}{1-(2x-1)^2}\right)&x\in[0,1]\\
           0& \mbox{else}.
          \end{cases}
\end{align*}

\begin{figure}[htb]
	\centering
	\begin{subfigure}[b]{0.45\textwidth}
		\centering 
		\def\svgwidth{150pt} 
		\begingroup%
		\makeatletter%
		\providecommand\color[2][]{%
			\errmessage{(Inkscape) Color is used for the text in Inkscape, but the package 'color.sty' is not loaded}%
			\renewcommand\color[2][]{}%
		}%
		\providecommand\transparent[1]{%
			\errmessage{(Inkscape) Transparency is used (non-zero) for the text in Inkscape, but the package 'transparent.sty' is not loaded}%
			\renewcommand\transparent[1]{}%
		}%
		\providecommand\rotatebox[2]{#2}%
		\ifx\svgwidth\undefined%
		\setlength{\unitlength}{259.99999237bp}%
		\ifx\svgscale\undefined%
		\relax%
		\else%
		\setlength{\unitlength}{\unitlength * \real{\svgscale}}%
		\fi%
		\else%
		\setlength{\unitlength}{\svgwidth}%
		\fi%
		\global\let\svgwidth\undefined%
		\global\let\svgscale\undefined%
		\makeatother%
		\begin{picture}(1,0.59615388)%
		\put(0,0){\includegraphics[width=\unitlength,page=1]{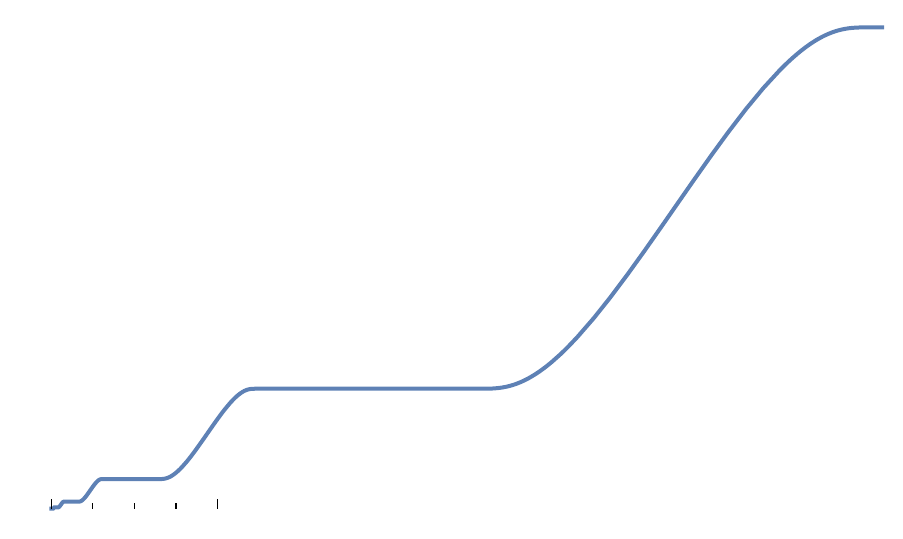}}%
		\put(0.21040609,-0.0171212){\color[rgb]{0,0,0}\makebox(0,0)[lb]{\smash{$0.2$}}}%
		\put(0,0){\includegraphics[width=\unitlength,page=2]{df_plot.pdf}}%
		\put(0.39950724,-0.0171212){\color[rgb]{0,0,0}\makebox(0,0)[lb]{\smash{$0.4$}}}%
		\put(0,0){\includegraphics[width=\unitlength,page=3]{df_plot.pdf}}%
		\put(0.58153694,-0.02015165){\color[rgb]{0,0,0}\makebox(0,0)[lb]{\smash{$0.6$}}}%
		\put(0,0){\includegraphics[width=\unitlength,page=4]{df_plot.pdf}}%
		\put(0.76760724,-0.0171212){\color[rgb]{0,0,0}\makebox(0,0)[lb]{\smash{$0.8$}}}%
		\put(0,0){\includegraphics[width=\unitlength,page=5]{df_plot.pdf}}%
		\put(0.95367755,-0.02116181){\color[rgb]{0,0,0}\makebox(0,0)[lb]{\smash{$1.0$}}}%
		\put(0,0){\includegraphics[width=\unitlength,page=6]{df_plot.pdf}}%
		\put(-0.05151778,0.2032397){\color[rgb]{0,0,0}\makebox(0,0)[lb]{\smash{$0.05$}}}%
		\put(0,0){\includegraphics[width=\unitlength,page=7]{df_plot.pdf}}%
		\put(-0.05353808,0.38355278){\color[rgb]{0,0,0}\makebox(0,0)[lb]{\smash{$0.10$}}}%
		\put(0,0){\includegraphics[width=\unitlength,page=8]{df_plot.pdf}}%
		\put(-0.05252793,0.56487447){\color[rgb]{0,0,0}\makebox(0,0)[lb]{\smash{$0.15$}}}%
		\put(0,0){\includegraphics[width=\unitlength,page=9]{df_plot.pdf}}%
		\end{picture}%
		\endgroup%
		\caption{The monotone operator $f'$ on $[0,1]$.} 
	\end{subfigure}	
	\begin{subfigure}[b]{0.45\textwidth}
		\centering 
		\def\svgwidth{150pt} 
			\begingroup%
			\makeatletter%
			\providecommand\color[2][]{%
				\errmessage{(Inkscape) Color is used for the text in Inkscape, but the package 'color.sty' is not loaded}%
				\renewcommand\color[2][]{}%
			}%
			\providecommand\transparent[1]{%
				\errmessage{(Inkscape) Transparency is used (non-zero) for the text in Inkscape, but the package 'transparent.sty' is not loaded}%
				\renewcommand\transparent[1]{}%
			}%
			\providecommand\rotatebox[2]{#2}%
			\ifx\svgwidth\undefined%
			\setlength{\unitlength}{259.99999237bp}%
			\ifx\svgscale\undefined%
			\relax%
			\else%
			\setlength{\unitlength}{\unitlength * \real{\svgscale}}%
			\fi%
			\else%
			\setlength{\unitlength}{\svgwidth}%
			\fi%
			\global\let\svgwidth\undefined%
			\global\let\svgscale\undefined%
			\makeatother%
			\begin{picture}(1,0.60384616)%
			\put(0,0){\includegraphics[width=\unitlength,page=1]{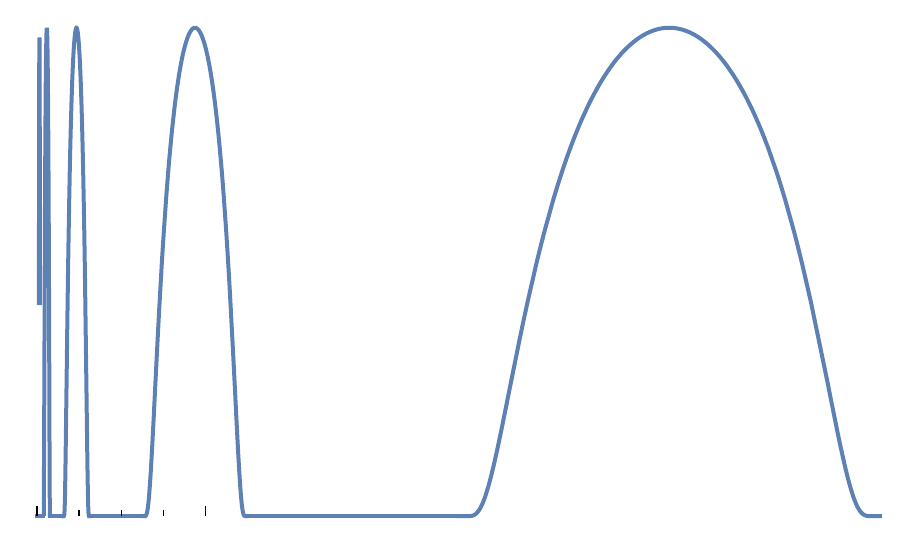}}%
			\put(0.19926148,-0.01167175){\color[rgb]{0,0,0}\makebox(0,0)[lb]{\smash{$0.2$}}}%
			\put(0,0){\includegraphics[width=\unitlength,page=2]{h_plot.pdf}}%
			\put(0.38804602,-0.01167175){\color[rgb]{0,0,0}\makebox(0,0)[lb]{\smash{$0.4$}}}%
			\put(0,0){\includegraphics[width=\unitlength,page=3]{h_plot.pdf}}%
			\put(0.57683175,-0.01167175){\color[rgb]{0,0,0}\makebox(0,0)[lb]{\smash{0$.6$}}}%
			\put(0,0){\includegraphics[width=\unitlength,page=4]{h_plot.pdf}}%
			\put(0.76359709,-0.01167175){\color[rgb]{0,0,0}\makebox(0,0)[lb]{\smash{$0.8$}}}%
			\put(0,0){\includegraphics[width=\unitlength,page=5]{h_plot.pdf}}%
			\put(0.95137263,-0.01167175){\color[rgb]{0,0,0}\makebox(0,0)[lb]{\smash{$1.0$}}}%
			\put(0,0){\includegraphics[width=\unitlength,page=6]{h_plot.pdf}}%
			\put(-0.0383858,0.16838184){\color[rgb]{0,0,0}\makebox(0,0)[lb]{\smash{$0.1$}}}%
			\put(0,0){\includegraphics[width=\unitlength,page=7]{h_plot.pdf}}%
			\put(-0.04141625,0.31646814){\color[rgb]{0,0,0}\makebox(0,0)[lb]{\smash{$0.2$}}}%
			\put(0,0){\includegraphics[width=\unitlength,page=8]{h_plot.pdf}}%
			\put(-0.0404061,0.46455523){\color[rgb]{0,0,0}\makebox(0,0)[lb]{\smash{$0.3$}}}%
			\put(0,0){\includegraphics[width=\unitlength,page=9]{h_plot.pdf}}%
			\put(0.8061017,0.38486812){\color[rgb]{0,0,0}\makebox(0,0)[lb]{\smash{}}}%
			\end{picture}%
			\endgroup%
		\caption{The function $h=f''$ on $[0,1]$.} 
	\end{subfigure}
\caption{Construction of the convex function $f$ in \eqref{eq:example}.}\label{fig:example}
\end{figure}

The function $h$ is nonnegative and infinitely differentiable on $\mathbb{R}\setminus\{0\}$ so that 
$\int_0^y h(z)\,dz$ is continuous nondecreasing. It follows that $f$ is convex on $\mathbb{R}$ and infinitely differentiable on 
$\mathbb{R}\setminus\{0\}$ and hence satisfies properties \eqref{def:F functions a}-\eqref{def:F functions b} of Definition~\ref{def:F functions}.  
The function $f$, however does not satisfy \eqref{def:F functions c} of Definition~\ref{def:F functions} as it is affine on every interval 
$(2^{-(2n+2)}, 2^{-(2n+1)})$ for $n\in\mathbb{N}_0$ fixed with slope $a_n$ given by
\begin{align*}
a_n=  \int_0^{2^{-2(n+1)}} h(z) \,dz &=  \sum_{j=n+1}^\infty\int_{\mathbb{R}}\psi\left(  2^{2j+1}x-1 \right)\,dx\\
&= \sum_{j=n+1}^\infty2^{-(2j+1)}\int_{\mathbb{R}}\psi\left( y\right)\,dy. 
\end{align*}
Now, since $f$ is affine on infinitely many intervals in $[0,1/2]$, its subdifferential, $\partial f$, is constant and singleton on infinitely 
many intervals in $[0,1/2]$ with value on these intervals given by $\{a_n\}_{\mathbb{N}_0}$. It follows that 
$\partial f^*=(\partial f)^{-1}$ is multi-valued at each point in $\{a_n\}_{\mathbb{N}_0}$, of which there are infinitely 
many, and therefore $f^*$ cannot be in $\mathcal{F}$.\qed
\end{example}

\section{Examples and Illustrations}\label{s:examples}
 In this section we detail a number of computational examples and applications which 
utilize the class of monotone operators introduced above. We perform our symbolic 
computations in \emph{Maple} making use of the data-structures provided by the 
\emph{Symbolic Convex Analysis Toolkit (SCAT)} developed by Borwein 
\& Hamilton \cite{borwein2009symbolic}. We shall also make use of an additional function, shown in Figure~\ref{fig:inverse function}, for computing the inverse of a monotone operator. The source code for the examples which follow as well as the SCAT library are available online at:
  \begin{center}
  	\href{http://vaopt.math.uni-goettingen.de/software.php}{http://vaopt.math.uni-goettingen.de/software.php}
  \end{center}

Although we consider the one-dimensional setting, it is worth noting that, as recognized by 
\cite{bauschke1997fenchel}, separable convex functions on $\mathbb{R}^n$ can still be handled. 
Recall that a convex function $f:\mathbb{R}^n\to(-\infty,+\infty]$ is \emph{separable} if there exist 
convex functions $f_j:\mathbb{R}\to(-\infty,+\infty]$ such that $f(x) = \sum_{j=1}^nf_j(x_j).$  For such a function, 
 $$\partial f(x) = \partial f_1(x_1)\times \dots\times \partial f_n(x_n),\qquad f^*(y)=\sum_{j=1}^nf_j^*(y_j).$$
In this way, we may also consider monotone operators $T\colon\mathbb{R}^n\setto\mathbb{R}^n$ such that 
 $$T(x)=T_1(x_1)\times \dots \times T_n(x_n)$$
where each $T_j\colon\mathbb{R}\setto\mathbb{R}$ is a monotone operator. We give examples 
in which this kind of structure arising in Sections~\ref{sec:prox from f} and \ref{sec:recovery}.

\begin{figure}[h]
\begin{framed}\ttfamily
	{\color{mapleCode}
\begin{verbatim}
with(SCAT): # load the SCAT package 
# Compute the inverse of a SD-type object (i.e. a cyclic monotone operator) 
# using Fenchel conjugation
Invert := proc(sdf::SD,y)  
   local sdg, f, conjf; 
   f := Integ(sdf); # antiderivative of sdf
   g := Conj(f,y);  # Fenchel conjugate of f
   sdg := SubDiff(g,y); # subderivative of g
   return sdg; 
end proc:
\end{verbatim}\vspace{-1em}}
\end{framed}\vspace{-1em}
\caption{Inversion of a maximal cyclic monotone operator using Fenchel conjugation.}\label{fig:inverse function}
\end{figure}


\subsection{Explicit Formula for Proximity Operators}\label{sec:prox from f}
Recall that the \emph{proximity operator} of a proper, lsc, convex function 
$f:\mathbb{R}^n\to(-\infty,+\infty]$ with parameter $\lambda>0$ is given by 
 \begin{equation*}
 \prox_{f}^\lambda:=\argmin_{y\in\mathbb{R}}\left\{f(y)+\frac{1}{2\lambda}\|\cdot-y\|^2\right\}.
 \end{equation*}

Proximity operators are the building blocks of many iterative algorithms in optimization 
and thus it is important, in practice, that they can be efficiently computed.
One such possibility is to find an \emph{explicit formula} for the proximity operator. 

We compute explicit forms for some classical proximity operators using the 
framework of $\mathcal{T}$-operators. Our approach exploits the formula
  \begin{equation}\label{eq:prox as resolvent}
	  \prox_{f}^\lambda = (I+\lambda \partial f)^{-1}.
	\end{equation}
Given a function $f\in\mathcal{F}$, we symbolically compute its subdifferential using the SCAT 
package. Using \eqref{eq:prox as resolvent} and Proposition~\ref{prop:T operator properties}, we deduce 
that $\prox_{f}^\lambda\in\mathcal{T}$.

Two example computations of proximity operators are in given in Figures~\ref{fig:l1 prox} and \ref{fig:interval proj}. Furthermore, it is worth nothing that, thanks to our theory in Section~\ref{sec:operators in 1D}, this is computation actually proves the resulting formula for the proximity function.

\begin{figure}[tbh]
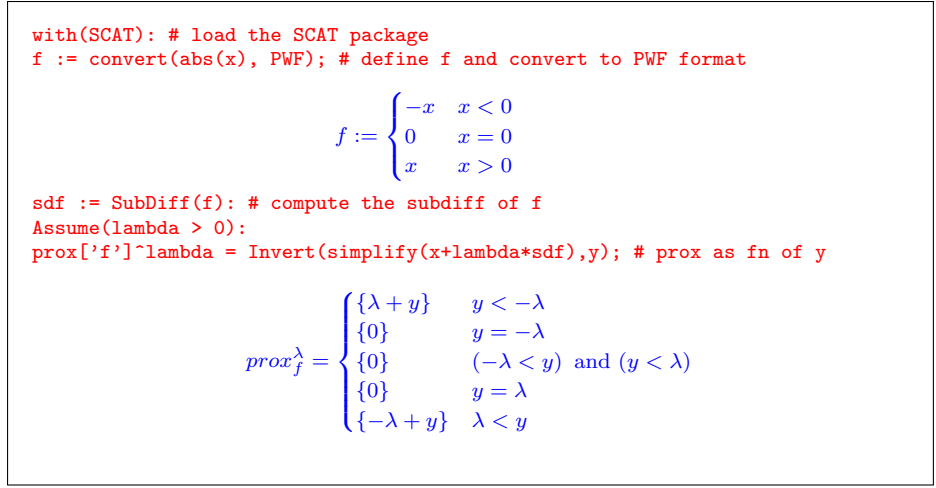

\begin{framed}\ttfamily\noindent
	{\color{mapleCode}
		with(SCAT): \# load the SCAT package \\
		f := convert(abs(x), PWF); \# define f and convert to PWF format \\
		$${\color{mapleMath}
				  f:=
				   \begin{cases}
					-x & x<0 \\
					0     & x=0 \\
					x & x>0 \\
					\end{cases}\qquad\qquad
		}$$   
		sdf := SubDiff(f): \# compute the subdiff of f \\
		Assume(lambda > 0): \\
		prox['f']\textasciicircum lambda = Invert(simplify(x+lambda*sdf),y); \# prox as fn of y
		$${\color{mapleMath}
		  prox_f^{\lambda} =
		   \begin{cases}
			\{\lambda+y\} & y<-\lambda \\
			\{0\}     & y=-\lambda \\
			\{0\}     & (-\lambda<y)\text{ \normalfont and }(y<\lambda) \\
			\{0\}     & y=\lambda \\
			\{-\lambda+y\} & \lambda<y \\
			\end{cases}
		}$$    
	}
\end{framed}\vspace{-1em}
\caption{Computation of the proximity function of $f=\|\cdot\|_1$.}\label{fig:l1 prox}
\end{figure}

\begin{figure}[tbh]
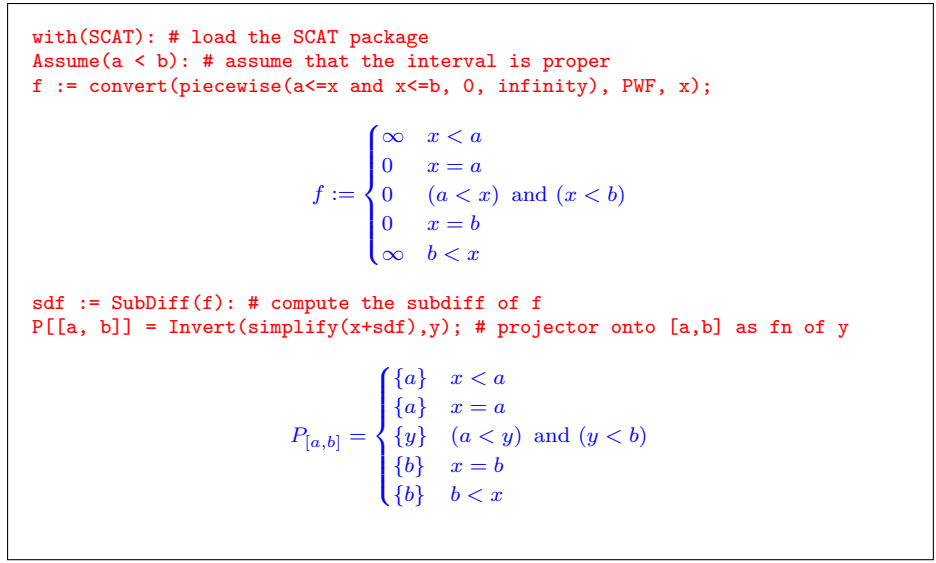

\begin{framed}\ttfamily\noindent
	{\color{mapleCode}
	with(SCAT): \# load the SCAT package \\
	Assume(a < b): \# assume that the interval is proper \\
	f := convert(piecewise(a<=x and x<=b, 0, infinity), PWF, x); 
	 $${\color{mapleMath} f := \begin{cases}
	           \infty & x<a \\
	           0 & x=a \\
	           0 & (a<x)\text{ \normalfont and }(x<b) \\
	           0 & x=b \\
	           \infty & b<x \\
             \end{cases}}$$
	\noindent sdf := SubDiff(f): \# compute the subdiff of f\\
	P[[a, b]] = Invert(simplify(x+sdf),y); \# projector onto [a,b] as fn of y
	{\color{mapleMath}
		 $$  P_{[a,b]} = \begin{cases}
		           \{a\} & x<a \\
		           \{a\} & x=a \\
		           \{y\} & (a<y)\text{ \normalfont and }(y<b) \\
		           \{b\} & x=b \\
		           \{b\} & b<x \\
	             \end{cases}$$
		}	
	}
\end{framed}
\vspace{-1em}
\caption{Computation of the proximity function of $f=\iota_{[a,b]}$ ({\em i.e.,} projector onto $[a,b]$).}\label{fig:interval proj}
\end{figure}

\subsection{Recovery of Penalty Functions}\label{sec:recovery}
 Given a monotone operator $T:\mathbb{R}^n\setto\mathbb{R}^n$, we consider the problem 
of finding a so-called \emph{penalty function} $f:\mathbb{R}^n\to(-\infty,+\infty]$, that is, a 
function $f$ whose subdifferential can be identified with $T$. Precisely, find a function $f$ such that
   $$\gph T\subseteq\gph(\prox_f^\lambda).$$
 We shall focus on the case in which $\lambda=1$ as it covers all the technicalities of the general 
case and thus we define $\prox_f:=\prox_f^1$. The same problem was previously studied 
in \cite{bayram2015penalty}, without symbolic computational tools.
 
The following proposition and its proof shall form the basis of our approach.
\begin{proposition}[Recovery of penalty functions]\label{prop:penalty from prox}
	Let $T:\mathbb{R}^n\setto\mathbb{R}^n$ be a maximal cyclically monotone operator. 
There exists a proper, lsc function $f$ with $f+\frac{1}{2}\|\cdot\|^2$ convex such that $T=\prox_f$. 
Furthermore, if $T\in\mathcal{F}$ then there exists an $f$ such that 
$f+\frac{1}{2}\|\cdot\|^2$ belongs to $\mathcal{F}$.
\end{proposition}
\begin{proof}
	By Fact~\ref{fact:cyclically monotonicity}, there exists a proper, lsc, convex function 
$h:\mathbb{R}^n\to(-\infty,+\infty]$ such that $T=\partial h$, and in particular, if 
$T\in\mathcal{T}$ then Theorem~\ref{th:F and T} ensures that we may choose $h\in\mathcal{F}$. 
Using Fermat's rule \cite[Th.~10.1]{VA}, we deduce that
	\begin{align*}
	T(x) = \partial h(x)
	= \left(\partial h^*\right)^{-1}(x) 
	&= \{y\in\mathbb{R}^n:x\in\partial h^*(y)\} \\
	&= \argmin_{y\in\mathbb{R}^n}\{h^*(y)-\langle x,y\rangle\} \\
	&= \argmin_{y\in\mathbb{R}^n}\left\{\left(h^*(y)-\frac{1}{2}\|y\|^2\right)+\frac{1}{2}\|x-y\|^2\right\}.
	\end{align*}
	We therefore have $T=\prox_f$ where $f:=h^*-\frac{1}{2}\|\cdot\|^2$. The fact 
that $f$ is proper and lsc with $f+\frac{1}{2}\|\cdot\|^2$ convex follows since $h^*$ is proper, lsc and 
convex. In particular, if $h\in\mathcal{F}$ then $f+\frac{1}{2}\|\cdot\|^2=h^*\in\mathcal{F}$. \qed
\end{proof}

We are now ready to state our strategy for reconstruction of the penalty function associated with 
the monotone operator $T\in\mathcal{T}$. 
    \begin{enumerate}[(i)]
    \itemsep0em 
    \item\label{item:i} Find a maximal cyclically monotone extension of $\widetilde{T}\in\mathcal{T}$ of $T$.
    \item\label{item:ii} Find a function $h\in\mathcal{F}$ such that $\partial h=\widetilde{T}$.
    \item\label{item:iii} Compute the Fenchel conjugate $h^*$ of $h$.
    \item\label{item:iv} The penalty function $f$ can now be given as $f:=h^*-\frac{1}{2}\|\cdot\|^2$.
    \end{enumerate}
  We note that the existence of a function $h$ in Step~\eqref{item:ii} is possible due to
 Theorem~\ref{th:F and T} and can be obtained via integrating any selection of $T$ 
\cite[Prop.~1.6.1]{niculescu2006convex}. Step~\eqref{item:iii} can be performed for 
$\mathcal{F}$-functions within the SCAT package and Step~\eqref{item:iv} is clearly 
straightforward. Thus the only potentially difficult computation arises in Step~\eqref{item:i}; but 
this can sometimes be dealt with satisfactorily as we shall now demonstrate.

\begin{example}[Hidden convexity of the hard thresholding operator]
 The \emph{hard thresholding operator} $H_\alpha:\mathbb{R}\to\mathbb{R}$ for 
parameter $\alpha>0$ is defined by
  $$ H_\alpha(x) := \begin{cases}
                     x &  \text{if }|x|>\alpha, \\
                     0 & \text{otherwise.} \\
                     \end{cases} $$
 In compressive sensing, $H_\alpha$ is usually viewed as a selection of the (set-valued) 
proximity operator of the $\ell_0$-functional; a non-convex object. More precisely, 
   $$\prox_{\frac{\alpha}{2}\|\cdot\|_0}(x) =\argmin_{y\in\mathbb{R}}
\left\{\alpha\|y\|_0 +|x-y|^2\right\}=\begin{cases}
                      x & \text{if }|x|>\alpha,\\
                      \{0,x\} & \text{if }|x|=\alpha,\\
                      0    & \text{otherwise};
                    \end{cases}$$             
 and hence that $\gph H_\alpha\subseteq\gph\prox_{\frac{\alpha}{2}\|\cdot\|_0}$. 
Whilst both $H_\alpha$ and $\prox_{\frac{\alpha}{2}\|\cdot\|_0}$ are monotone operators, 
neither are maximal. Nevertheless, on account of having full domain, their unique maximal 
monotone extension can be given
   $$ T(x) := \begin{cases}
                x     & \text{if }|x|>\alpha,\\
                [0,x] & \text{if }|x|=\alpha,\\
                0     & \text{otherwise};
              \end{cases} $$
 and, moreover, it is easy to verify that $T$ belongs to $\mathcal{T}$. 
 
We are now in a position to recover the penalty function $f$ associated with $T$.
\begin{figure}[h]
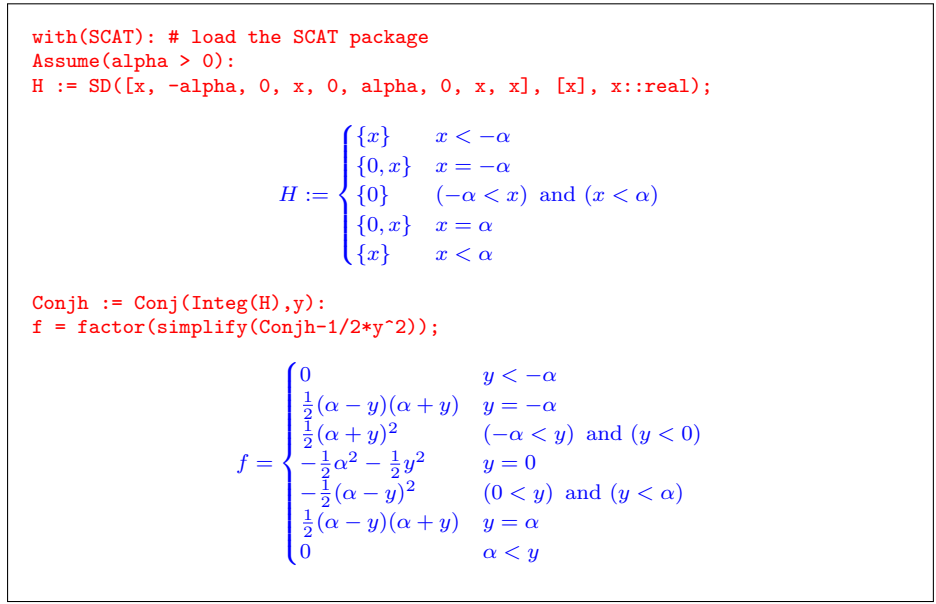

 \begin{framed}\ttfamily\noindent
 {\color{red}
 with(SCAT): \# load the SCAT package \\
 Assume(alpha > 0): \\
 H := SD([{x}, -alpha, {0, x}, {0}, alpha, {0, x}, {x}], [x], {x::real});
 }
 $${\color{blue} H:=\begin{cases}
         \{x\}   & x<-\alpha \\
         \{0,x\} & x=-\alpha \\
         \{0\}   & (-\alpha<x)\text{ \normalfont and }(x<\alpha) \\
         \{0,x\} & x=\alpha \\
         \{x\}   & x<\alpha \\
       \end{cases} }$$
 {\color{red}\noindent
 Conjh := Conj(Integ(H),y): \\
 f = factor(simplify(Conjh-1/2*y\textasciicircum 2));
 }
 
 $${\color{blue} 
    f=\begin{cases}
        0 & y<-\alpha \\
        \frac{1}{2}(\alpha-y)(\alpha+y) & y=-\alpha \\
        \frac{1}{2}(\alpha+y)^2 & (-\alpha<y)\text{ \normalfont and }(y<0) \\
        -\frac{1}{2}\alpha^2-\frac{1}{2}y^2 & y=0 \\
        -\frac{1}{2}(\alpha-y)^2 & (0<y)\text{ \normalfont and }(y<\alpha) \\
        \frac{1}{2}(\alpha-y)(\alpha+y) & y=\alpha \\
        0 & \alpha<y \\
       \end{cases} }$$
 \end{framed}
 \vspace{-1em}
 \caption{Recovery of a convex penalty associated with $H_\alpha$.}\label{fig:hard thresh}
 \end{figure}
 
 A closer look at the result from Figure~\ref{fig:hard thresh} shows that the penalty function $f$ 
may be expressed more concisely in the form
 \begin{equation}\label{eq:H penalty}
  f(y) = \begin{cases}
 0 & |y|>\alpha, \\
 -\frac{1}{2}\left(|y|-\alpha\right)^2  & |y|\leq\alpha.\\
 \end{cases}
 \end{equation}
It is also worth noting, that the entire procedure can be reserved to give a proof that the hard 
thresholding operator is monotone! More precisely, one should symbolically compute the proximity 
function of $f$ in \eqref{eq:H penalty} using the method in Section~\ref{sec:prox from f}. If the result is
 an extension of the original operator, $H_\alpha$, then it is necessarily monotone (as the 
subdifferential of a proper, lsc, convex function).
\end{example}

\subsection{Superexpectations, superdistributions and superquantiles} 
The class of functions $\mathcal T$ is well-suited for direct symbolic calculation of 
superexpectations, superdistributions and superquantiles as developed in \cite{OgrRus02,DenMar12,RocRoy14}.  
The {\em superexpectation} function, $\overline{\mathbf{E}}_X(x)$, 
of the random variable $X$ at level $x$ is defined as 
\begin{equation}\label{e:superexp}
\begin{split}
\mathbf{E}_X(x):= \mathbb{E}\ecklam{\max\{x, X\}} &= \int_{-\infty}^\infty\max\{x, x'\}\ dF_X(x')\\
&= \int_0^1\max\{x, Q_X(p)\}\ dp.
\end{split}
\end{equation}
where 
$\map{F_X}{\mathbb{R}}{[0,1]}$ is the {\em cumulative distribution function} of the random variable $X$
and  $\map{Q_X}{(0,1)}{(-\infty,+\infty)}$ is the {\em quantile function};  these are defined
respectively as 
\begin{align}\label{e:F_X}
 F_X(x)&:=\prob\klam{X\leq x}\\
\label{e:Q_X}
 Q_X(p)&:=\min\set{x}{F_X(x)\geq p}\qquad (p\in(0,1)).
\end{align}
The function $F_X$ is nondecreasing and right-continuous on $(-\infty, +\infty)$ with 
$$\lim_{x\to-\infty}F_X(x)=0,\quad \lim_{x\to+\infty}F_X(x)=1.$$
The maximal monotone extension of the distribution function $F_X$, denoted $\widetilde{\mathcal{F}}$ 
is called the {\em superdistribution} and 
is also generated by taking the subdifferential of the superexpectation function \cite[Th.~1]{RocRoy14}, 
which is a  finite convex function on $(-\infty, \infty)$ with 
 \begin{equation}\label{e:superexp-superdist1}
\gph\widetilde{\mathcal{F}}_X=\gph\partial \mathbf{E}_X, \quad F_X(x)=\mathbf{E}_X'^+(x),
 \end{equation}
 where
 \[
\mathbf{E}_X'^+(x):= \lim_{y\downarrow x}\frac{\mathbf{E}_X(y)-\mathbf{E}_X(x)}{y-x}.   
 \]
The superquantile function, denoted  $\widetilde{\mathcal{Q}}_X$  
is the maximal monotone extension of the quantile function and satisfies 
the inverse relationship \cite[Th.~2]{RocRoy14}
\[
   \paren{\gph\widetilde{\mathcal{F}}_X}^{-1} =\gph \partial \mathbf{E}^*_X = \gph\widetilde{\mathcal{Q}}_X,
\mbox{ and } Q_X(p)= \mathbf{E}^{*'-}_X(p).
\]
These objects are therefore amenable to symbolic convex analysis, via
subdifferentials of the superexpectation function $\mathbf{E}_X$ or its Fenchel conjugate. 
This provides a symbolic route to working with {\em coherent risk measures} such as 
{\em conditional valued-at-risk}. 

To demonstrate this approach, we symbolically derive an example which appears in \cite{RocRoy14}.
\begin{example}[Exponential distributions]
Let $X$ be exponentially distributed with parameter $\lambda>0$. That is, $X$ has cumulative 
distribution function $F_X=1-\exp(-\lambda x)$. Figure~\ref{fig:rockafellar} 
shows the symbolic computation of the superexpectations function, the superdistribution function, and the superquantile function of $X$. Note that, to compute the superexpection of $F_X$, we have made use of the fact that 
   $$ \lim_{x\to\infty}[E_X(x)-x]=0,$$
 which was proven as part of \cite[Th.~1]{RocRoy14}. In this way, it is not necessary to compute the potentially tricky `max' in the definition \eqref{e:superexp} directly.

\begin{figure}[h]
 \begin{framed}\ttfamily\noindent
 {\color{mapleCode}
  with(SCAT):\\
  F := 1-exp(-lambda*x): \# distribution fn of X \\
  Q := solve(F=p,x); \# quantile fn of X \\
  $${\color{mapleMath} 
	  Q := - \frac{\ln(1-p)}{\lambda}  
  }$$
  superQ := factor(1/(1-p)*int(subs(p=t,Q),t=p..1)); \# superquantile fn of X \\
  $${\color{mapleMath} 
  	  superQ := - \frac{\ln(1-p)-1}{\lambda}
  }$$
  Assume(lambda>0): \\
  F := convert(piecewise(x >= 0,F), SD,x);
  $${\color{mapleMath} 
  	  E := \begin{cases}
  	          \{0\} & x<0 \\
  	          \{0\} & x=0 \\
  	          \{1-e^{-\lambda x}\} & 0 < x \\
   	  \end{cases}
  }$$  
  \# compute the superexpection function of F \\
  E0 := Integ(F,x): \\
  c0 := Eval(simplify(E0-x), x = infinity): \\
  E  := simplify(E0 - c0); \\
  $${\color{mapleMath} 
  	  E := \begin{cases}
  	          \displaystyle\frac{1}{\lambda} & x<0 \\[1.25ex]
  	          \displaystyle\frac{1}{\lambda} & x=0 \\[1.25ex]
  	          \displaystyle\frac{\lambda x+e^{-\lambda x}}{\lambda} & 0<x \\[1.25ex]
   	       \end{cases}
  }$$    
  conjE := conjE(E,p,x); \# the conjugate of the superexpectation
  $${\color{mapleMath} 
  	  conjE := \begin{cases}
  	          \infty & p < 0 \\[1.15ex]
  	          -\displaystyle\frac{1}{\lambda} & p=0 \\[1.25ex]
  	          -\displaystyle\frac{(-1+p)(\ln(1-p)-1)}{\lambda} & (0<p)\text{ \normalfont and }(p<1) \\[1.25ex]
  	          0 & p=1 \\
  	          \infty & 1<p \\
   	       \end{cases}
  }$$    
 }
 \end{framed}
 \vspace{-1em}
 \caption{Computation of super-functions for the exponential distribution.}\label{fig:rockafellar}
 \end{figure}
\end{example}

\begin{acknowledgements}
DRL was supported in part by Deutsche Forschungsgemeinschaft Collaborative Research Center SFB755. MKT was supported by Deutsche Forschungsgemeinschaft RTG2088. 
\end{acknowledgements}


\end{document}